\documentclass{amsart}
\usepackage{latexsym,amscd,amssymb,amsopn,amsthm,amsfonts,amsmath}
\usepackage{color}
\usepackage[dvipsnames]{xcolor} 
\copyrightinfo{2010}{American Mathematical Society}
\newtheorem{theorem}{Theorem}[section]
\newtheorem{lemma}[theorem]{Lemma}
\newtheorem{corollary}[theorem]{Corollary}

\theoremstyle{definition}

\newtheorem{example}[theorem]{Example}

\theoremstyle{remark}

\numberwithin{equation}{section}

\begin{document}

\title[On Grothendieck type duality]
{On Grothendieck type duality 
for the space of holomorphic functions of several variables}

\author[Yu. Khoryakova]{Yulia Khoryakova}
\email{ykhoryakova@sfu-kras.ru}

\author[A. Shlapunov]{Alexander Shlapunov}
\email{ashlapunov@sfu-kras.ru}

\address{Siberian Federal University,
         Institute of Mathematics and Computer Science,
         pr. Svobodnyi 79,
         660041 Krasnoyarsk,
         Russia}



\begin{abstract}
We describe the strong dual space $({\mathcal O} (D))^*$ for the space ${\mathcal O} (D)$ of 
holomorphic functions of several complex variables over a bounded Lipschitz domain $D$ with connected 
boundary $\partial D$ (as usual, ${\mathcal O} (D)$ is endowed with the topology of the 
uniform convergence on the compact subsets of $D$). We identify the dual space with a closed 
subspace of the space of harmonic functions on the closed set ${\mathbb C}^n\setminus D$, 
$n>1$, with elements vanishing at the infinity and satisfying the 
tangential Cauchy-Riemann equations on $\partial D$. In particular, 
we extend in a way the classical  Grothendieck-K{\"o}the-Sebasti\~{a}o e Silva 
duality for the space of holomorphic functions of one complex variable 
to  the multi-dimensional  situation. We use the Bochner-Martinelli kernel ${\mathfrak U}_n$ 
in ${\mathbb C}^n$, $n>1$, instead of the Cauchy kernel over the complex plane ${\mathbb C}$
and we prove that the duality holds true if and only if the space ${\mathcal O} (D)\cap 
H^1 (D)$ of the Sobolev holomorphic functions over $D$ is dense in  ${\mathcal O} (D)$. 
\end{abstract}

\maketitle

\section*{Introduction}
\label{s.0}

One of the first dualities in the spaces of holomprphic functions was discovered in 
1950-'s independently by A.Grothendieck \cite{Grot2}, 
G. K{\"o}the  \cite{Koth2} and  J. Sebasti\~{a}o e Silva \cite{Silva}, 
who described the strong dual  $({\mathcal O} (D))^*$ for the space of holomorphic functions 
${\mathcal O} (D)$ (endowed with the standard Frech\'et topology) 
in a bounded simply connected domain $D\subset {\mathbb C}$:
\begin{equation} \label{eq.dual.compl}
({\mathcal O} (D))^* \cong {\mathcal O} (\hat {\mathbb C} \setminus D)
\end{equation} 
where ${\mathcal O} (\hat {\mathbb C} \setminus D)$ is the space of holomorphic 
functions on neighborhoods of the  closed set ${\mathbb C} \setminus D$, vanishing at 
the infinity, endowed with the standard  inductive limit topology. 
Unfortunately, the Hartogs' Theorem on the absence of compact singularities and 
the classical Liouville Theorem do not allow direct analogues of this duality for holomorphic 
functions of several complex variables (though, some generalizations in this direction 
were done for linearly convex multidimensional domains, see \cite{Aize3}, \cite{Mart1}). 

Alternative dualities were independently obtained by L.A.~Aizenberg and 
S.~G. Gindikin \cite{AizeGind1}, L.A.~Aizenberg and B.S.~Mityagin \cite{AizeMit} 
 and E.L.~Stout \cite{Stou1}: 
\begin{equation} \label{eq.dual.self}
({\mathcal O} (D) )^*\cong {\mathcal O} (\overline D)
\end{equation} 
if $D \subset {\mathbb C}^n$ is a bounded domain with the real analytic (pseudo-convex 
for $n>1$) boundary; actually different pairings were used at the core of their dualities 
(namely the ones generated by the inner products in the Bergmann space and Hardy space, 
respectively, cf. \cite{NaciShlaTark1}, \cite{Zorn1} for other possibilities). 

Many generalizations of the mentioned above dualities came from the theory of the Dolbeault 
co-homologies, see  \cite{Mart2}, \cite{Se}, and the theory of elliptic systems 
of partial differential equations, see  \cite{Bla}, \cite{Grot1}, \cite{NaciShlaTark1}, 
\cite{NakSa}, \cite{Nap}, \cite{ShDualSMJ}, \cite{ShTaDBRK}. 

One of the most general results, describing the duality for the spaces 
of solutions to elliptic differential operators with the topology of uniform 
convergence on compact sets,  belong to A. Grothendieck, see  \cite[Theorems 3 and 4]{Grot1}; 
it is similar in a way to \eqref{eq.dual.compl}. Unfortunately, the full description of the 
dual space was obtained for elliptic operators admitting bilateral regular fundamental 
solutions, see \cite[Theorem 4]{Grot1} that does not applicable to overdetermined elliptic 
operators like the multi-dimensional Cauchy-Riemann operator. Though 
\cite[Theorem 3]{Grot1} is still applicable for spaces of solutions to overdetermined 
elliptic operators  admitting left 
regular fundamental solutions, it gives the answer in terms of solutions to related 
underdetermined adjoint operators. The last type of spaces is too large to provide a 
topological isomorphism between the corresponding spaces (the theorem defines a linear 
continuous surjective mapping, only).
 
Another general scheme of producing dualities for (both determined and overdetermined) 
elliptic systems  was presented in \cite{ShTaDBRK}. It involves the concept of Hilbert space 
with reproducing kernel and the constructed pairings are closely related to the inner 
products of the used Hilbert spaces. However the application of the scheme  depends on the 
very subtle information regarding  the properties of the reproducing kernel that is not 
always at hands. 

To formulate the main result of the present paper,  let $D$ be a bounded domain in  
${\mathbb R}^{2n}$, $n>1$, and let ${\mathcal H} (\hat {\mathbb R}^{2n} \setminus  D)$ be the 
space of harmonic complex valued functions on the closed  set ${\mathbb R}^{2n} \setminus  D$, 
regular at the infinity, i.e. such that 
\begin{equation}
\label{eq.zero.infty}
\lim_{|x|\to + \infty}|u(x)| = 0.
\end{equation}
We endow the space with the standard inductive limit topology of harmonic functions
 on closed sets. 

Now let $\Sigma ( \hat {\mathbb C}^n \setminus  D)$ stand 
for  the closed subspace of the space ${\mathcal H} (\hat {\mathbb R}^{2n} \setminus  D)$
of functions $v \in {\mathcal H} (\hat {\mathbb R}^{2n} \setminus  D)$ 
satisfying the tangential Cauchy-Riemann equations on $\partial D$. 
For alternative descriptions of the space $\Sigma ( \hat {\mathbb C}^n \setminus  D)$ 
see Corollary \ref{c.Sigma} below.

\begin{theorem} \label{t.dual.compl.over.G}
Let $n>  1$ and let $D$ be a bounded domain in ${\mathbb C}^n$ with connected Lipschitz 
boundary, such that  the space $H^1 (D) \cap {\mathcal O} (D) $ of  holomorphic functions 
in $D$ from the Sobolev class $H^1 (D)$  is dense in ${\mathcal O} (D) $. 
Then (topologically)
\begin{equation} \label{eq.dual.A.compl.over.G}
({\mathcal O} (D))^* \cong \Sigma ( \hat {\mathbb C}^n \setminus  D).
\end{equation}
\end{theorem}

The pairing, related to the duality \eqref{eq.dual.A.compl.over.G}
will be described in section \S \ref{s.proof}, see \eqref{eq.pair.A.compl.NShT}. 
We note also that for a bounded domain $D \Subset {\mathbb C}^n$ with real analytic 
boundary, such that the space ${\mathcal O}(\overline D)$  is dense  in ${\mathcal O}(D)$, 
Theorem \ref{t.dual.compl.over.G} can be extracted from the results of \cite{NaciShlaTark1}, 
where the duality \eqref{eq.dual.self} was proved for a very special pairing.

Of course, as $\mathcal O(\overline D) \subset H^1 (D) \cap {\mathcal O} (D)$, 
the famous Oka-Weil theorem implies that the approximation property, assumed 
in Theorem \ref{t.dual.compl.over.G}, is always fulfilled   
for strictly pseudo-convex domains in ${\mathbb C}^n$, $n>1$,  
see, for instance, \cite[Ch. 1, \S\S F, G]{GuRo}. 
We prove in \S \ref{s.misc} that the approximation property is also 
necessary for  the duality presented in Theorem \ref{t.dual.compl.over.G} to be true. 
It is also worth to note that the original duality \eqref{eq.dual.compl} for holomorphic 
functions of one variable does not need any restrictions on the smoothness of the curve 
$\partial D$ or on  convexity of the domain $D$.

\section{Preliminaries} 
\label{s.1}

Let ${\mathbb R}^n$, $n\geq 2$, be the Euclidean space with the coordinates 
$x=(x_1,x_2 , \dots x_n)$ and let ${\mathbb C}^n \cong {\mathbb R}^{2n}$ be the related 
$n$-dimensional complex space with the coordinates $z=(z_1,z_2 , \dots z_n)$, 
$z_j = x_j + \iota \, x_{n+j}$,  with the imaginary unit $\iota$.  
Let also  $D$ stand for a bounded domain (open connected set) 
in ${\mathbb C}^n$ with Lipschitz boundary $\partial D$. 
We consider complex valued functions over subsets of ${\mathbb C}^n$. 
As usual, for $s \in {\mathbb Z}_+$,  we use the notations 
$C^s (D)$ and $C^s (\overline D)$ for the spaces  
of $s$-times continuously differentiable functions over $D$ and $\overline D$,  
respectively. The Lebesgue and the Sobolev Hilbert spaces over $D$ will be denoted by 
$L^2 (D)$ and $H^s (D)$, respectively. Let also $H^{s} (\partial D)$, $0<s<1$, stand for the 
standard Sobolev-Slobodetskii spaces over $\partial D$.

Given any open set $U$ in ${\mathbb R}^n$, let   ${\mathcal H} (U)$
stand for the space of harmonic functions in $U$ with the
topology of uniform convergence on compact subsets of $U$. 
It is known that ${\mathcal H} (U)$ is a Fr\'{e}chet-Schwartz topological 
vector space ($F$-space),  see, for instance, 
\cite[Ch. II, \S 4]{Shaef}. Moreover, the a priori estimates for solutions to 
elliptic equations, see, for instance, \cite{GiTru83}, mean that ${\mathcal H} (U) \subset C^
\infty (U)$  is a closed subspace of the space  $C (U)$, i.e.  the topology of the space 
can be defined   both by a system of semi-norms $p_\nu (u) = \max_{x\in K_\nu} 
| u (x)|$ related to an increasing system of compact sets $\{ K_{\nu} \} \subset U$ 
satisfying $\cup_{\nu} K_\nu = U$, and a  system of semi-norms $p^{(\alpha)}_\nu (u) = 
\max\limits_{z\in K_\nu }  | \partial ^\alpha u (z)|$, $\alpha \in {\mathbb Z}^{n}_+$, 
additionally assuming that each $K_\nu$ is the closure of an open set in $U$. 

Any holomorphic function over an open set $U\subset {\mathbb C}^n$ is harmonic and we 
consider the space ${\mathcal O} (U)$ of holomorphic functions in $U$ as a closed
subspace of ${\mathcal H} (U)$.

Next, for a closed set $\sigma \subset {\mathbb R}^n$  denote by ${\mathcal H} (\sigma)$ 
the set of harmonic functions on various neighbourhoods of $\sigma$
depending on the function. Actually  ${\mathcal H} (\sigma)$ can be considered as the
space of (equivalence classes) of harmonic functions on $\sigma$; 
two such solutions are equivalent if there is a neighbourhood of $\sigma$ where
they are equal. In ${\mathcal H} (\sigma)$, a sequence $\{ u_{\nu} \}$ is said to converge if
there exists a neighbourhood $\mathcal V$ of $\sigma$ such that all the functions are
defined at least in $\mathcal V$ and converge uniformly on compact subsets of $\mathcal V$.
Alternatively the topological space ${\mathcal H} (\sigma)$ can be described as the inductive
limit of spaces ${\mathcal H} (U_{\nu})$, where $\{ U_{\nu} \}$ is any decreasing
sequence of open sets containing $\sigma$, such that each neighbourhood of
$\sigma$ contains some $U_{\nu}$, and such that each connected component of
each $U_{\nu}$ intersects $\sigma$. 
Thus, ${\mathcal H} (U_{\nu}) \rightarrow {\mathcal H} (\sigma)$
are one-to-one. Then ${\mathcal H} (\sigma)$ is necessarily a Hausdorff space. 
Actually, ${\mathcal H} (\sigma)$ is the so called $DF$-space, 
see \cite[Ch. II, \S 6]{Shaef}.

For a bounded domain $D \subset {\mathbb R}^{2n}$ 
we define the space ${\mathcal H} (\hat {\mathbb R}^{2n} \setminus D)$, $n>1$, as a closed 
subspace of  ${\mathcal H} ({\mathbb R}^{2n} \setminus D)$ satisfying \eqref{eq.zero.infty}.

Next, we recall that a function $w_0\in L^1 (\partial D)$ is called a $CR$-function on 
the hyper-surface $\partial D$ if it  satisfies the weak tangential 
Cauchy-Riemann equations on $\partial D$, i.e. 
\begin{equation} \label{eq.CR.weak}
\int_{\partial D} w_0 \, \overline \partial \psi = 0  
\end{equation}
for all $(n,n-2)$-differential forms $\psi$ with the coefficients of the class 
$C^1 (\overline D)$, 
see, for instance, \cite[Ch. 2, \S 6]{Ky}.  Then the space 
$\Sigma (\hat {\mathbb C}^n \setminus D) $, defined above 
as the set of elements $v \in {\mathcal H} (\hat {\mathbb R}^{2n} \setminus  D)$ 
satisfying the tangential Cauchy-Riemann equations on $\partial D$,  is a closed subspace 
of the space ${\mathcal H} (\hat {\mathbb C}^n \setminus D)$  because any sequence, 
converging in ${\mathcal H} (\hat {\mathbb C}^n \setminus D)$, converges uniformly 
on the compact set $\partial D \subset {\mathbb C}^n$. 

Let us give a characterization of the space
$\Sigma (\hat {\mathbb C}^n \setminus D) $ for Lipschitz domains. 

With this purpose, denote by $\Delta_n $  the usual Laplace operator $\sum_{j=1}^n \frac{\partial^2}{\partial x _j^2}  $ in the Euclidean space ${\mathbb R}^n$.  
It is well-known, that the Laplace operator admits a fundamental solution $\Phi_{n}$ of the 
convolution type:
\begin{equation*} 
\Phi_{n} (y-x) = 
\left\{ 
\begin{array}{lll}
\frac{|y-x|^{2-n}}{(2-n)\sigma_n} &  n\geq 3,\\
\frac{\ln{|y-x|}}{2\pi} &  n = 2,\\
\end{array}
\right.
\end{equation*}
where $\sigma_n$ is the square of the unit sphere in ${\mathbb R}^n$. 

Let $\overline \partial$ denote the Cauchy-Riemann operator in ${\mathbb C}^n$, 
i.e. it is $n$-column with the components 
$\overline \partial_j  = 
\frac{1}{2}\Big(\frac{\partial }{\partial x_j} + \iota \frac{\partial }{\partial x_{j+n}}
\Big)$.  

Next, denote by 
$$
{\mathfrak U}_n (z,\zeta) = \frac{(n-1)}{(2\pi \iota)^n} \sum_{j=1}^n
\frac{(-1)^{j-1}(\overline \zeta_j - \overline z_j)}
{|\zeta-z|^{2n}} d\overline \zeta[j] \wedge 
d\zeta, \, z,\zeta \in {\mathbb C}^n, \, z\ne \zeta ,
$$ 
the Bochner-Martinelli kernel in the complex space ${\mathbb C}^n$,  see, for instance, 
\cite{Ky}. As it is known the Bochner-Martinelli kernel ${\mathfrak U}_n$ can be presented 
as 
\begin{equation} \label{eq.MB.*}
{\mathfrak U}_n (\zeta,z)  = \sum_{j=1}^n \Big( \overline \partial ^*_{j,\zeta} \Phi_{2n} 
(\zeta,z)\Big)(-1)^{j-1}d\overline \zeta [j] \wedge d\zeta 
\end{equation}
where $\zeta = (\zeta_1, \dots \zeta_n )$, $\zeta _j= y_j + \iota y_{j+n}$, and 
$\overline \partial_j ^* = 
\frac{1}{2}\Big(\frac{\partial }{\partial y_j} - \iota \frac{\partial }{\partial y_{j+n}}
\Big)$ are the components  of the formal adjoint operator  
$\overline \partial ^* = (\overline \partial_1 ^*, \dots \overline \partial_n ^*)$ 
for the Cauchy-Riemann operator $\overline \partial$, see \cite[\S 1]{Ky}. 
In particular, the Bochner-Martinelli kernel is harmonic with respect to 
$\zeta$ if $\zeta \ne z$ and $n>1$. Of course, for $n=1$ the kernel 
${\mathfrak U}_n (\zeta,z)$ coincides with the Cauchy kernel ${\mathfrak K} (\zeta,z)= 
\frac{1}{2\pi \iota} \frac{1}{\zeta-z}$; in this case it is holomorphic with respect to 
$\zeta$ if $\zeta \ne z$. 

Clearly, if $D$ is a domain with Lipschitz boundary then 
given a sufficiently regular function $u_0$  on $\partial D$, denote by 
$$
M_{\partial D} u_0 (z) = \int_{\partial D} {\mathfrak U}_n (z,\zeta)  u_0 (\zeta), 
\, z \not \in \partial D,
$$ 
its Bochner-Martinelli integral. Of course, the Bochner-Martinelli integral 
$M_{\partial D}u_0 (z)$ is well-defined for any $u_0 \in H^{1/2} (\partial D)$ as a parameter 
dependent integral if $z \not \in \partial D$. We denote by $M^-_{\partial D}u_0 $ its 
restriction to $D$ and by $M^+_{\partial D}u_0$ its restriction to ${\mathbb C}^n 
\setminus \overline D$. Actually it is a version 
of the double layer potential and so it induces the continuous linear operator 
\begin{equation} \label{eq.M-.cont}
M_{\partial D}^-:  H^{1/2} (\partial D) \to H^1 (D),
\end{equation} 
see, for instance, \cite[\S 16]{Ky}, 
\cite[\S 2.3.2.5 ]{ReSh}, for smooth domains or \cite{Cost88} for Lipschitz domains.
On the other hand, for any Lipschitz domain  $G \subset {\mathbb C}^n$ containing $\overline D$ the same arguments imply that 
the Bochner-Martinelli integral $M_{\partial D}$ induces continuous linear mappings  
$$
M^+_ {\partial  D} : H^{1/2} (\partial D) \to H^1 (G \setminus \overline D) \cap {\mathcal H} (G \setminus \overline D) .
$$ 
Next, by the structure of the kernel ${\mathfrak U}_n$ we have  
\begin{equation} \label{eq.M.infty}
|M^+_ {\partial  D} u_0 (z)| \leq c (\partial D) \|u_0\|_{H^{1/2} (\partial D) } |z|^{1-2n}
\end{equation}
and, therefore $M_{\partial D}$   induces 
continuous linear mapping  
\begin{equation} \label{eq.M+.cont}
M^+_ {\partial  D} : H^{1/2} (\partial D) \to H^{1}_{\rm loc} 
({\mathbb C}^n \setminus D) \cap {\mathcal H} (\hat {\mathbb C}^n \setminus \overline D) .
\end{equation}

In the following statement ${\rm t^-}$ and ${\rm t^+}$ stand 
for the continuous trace mappings
$$
{\rm t^-}: H^1 (D) \to H^{1/2} (\partial D) , \, 
{\rm t}^+: H^{1}_{\rm loc} ({\mathbb C}^n \setminus D)  \to H^{1/2} (\partial D). 
$$ 

The following statement can be easily extracted from  
\cite[Theorem 7.1 and Corollary 15.5]{Ky} for domains smooth boundaries. 

\begin{theorem} \label{t.Sigma}
Let $D\subset {\mathbb C}^n$  be a bounded Lipschitz domain.  
Then the following conditions are equivalent:
\begin{itemize}
\item[(1)]
The function $w_0\in H^{1/2} (\partial D)$ is a $CR$-function on $\partial D$;
\item[(2)]
there is a function $w \in H^1 (D) \cap {\mathcal O} (D)$ satisfying ${\rm t}^-(w)= 
w_0 $ on $\partial D$; 
\item[(3)]
$M^{+}_{\partial D}  w_0 \equiv 
 0$ in ${\mathbb C}^n \setminus \overline D$.
\end{itemize}
\end{theorem}

\begin{proof} The proof is based on the following statement that is well known 
for $C^1 (\partial D)$-functions and $C^1$-smooth domain $D$, see \cite[Corollary 15.5]{Ky}.

\begin{lemma} \label{l.Sigma}
Let $D\subset {\mathbb C}^n$  be a bounded Lipschitz domain. Given $w_0  
\in H^{1/2} (\partial D) $ there is a function $w \in H^1 (D) \cap {\mathcal O} (D)$ 
satisfying ${\rm t}^-(w)= v$ on $\partial D$ if and only if $M^{+}_{\partial D} w_0\equiv 
 0$. 
\end{lemma}

\begin{proof} As the operators $M^{\pm}_{\partial D}$, given by 
\eqref{eq.M-.cont}, \eqref{eq.M+.cont}, are continuous, then 
according to the Bochner-Martinelli formula for the Sobolev 
holomorphic function, see \cite[p. 166]{Ky} we have 
\begin{equation} \label{eq.Green.2}
M_{\partial D} {\rm t^-}w (z) 
=  \left\{
\begin{array}{ll} w(z) & x \in D, \\
0, & z \not \in \overline D, \\
\end{array}
\right.
\end{equation} 
for any $w \in  H^1 (D) \cap {\mathcal O} (D)$. 

In particular, this means that for a function $w_0 \in H^{1/2} (\partial D)$ 
we have $M^{+}_{\partial D} w_0 \equiv 0$ in ${\mathbb C}^n \setminus 
\overline D$ if there is a function $w \in H^1 (D) \cap {\mathcal O} (D)$ satisfying 
${\rm t}^-(w)= w_0$ on $\partial D$. 

Back, pick a function $w_0 \in H^{1/2} (\partial D)$ satisfying 
 $M^{+}_{\partial D} w_0 \equiv 0$ in ${\mathbb C}^n \setminus 
\overline D$. 

Denote by $\overline \partial_\nu$ the so-called 
 "complex normal derivative"{} with respect to $\partial D$ :
$$
\overline \partial_\nu w = \  \sum_{j=1}^n  (\nu _j - \iota \nu_{j+n})
\overline \partial _j w 
$$
where $\nu (z) = (\nu_1 (\zeta), \dots, \nu_{2n} (\zeta))$ is the unit outward normal 
to $\partial D$ at the point $\zeta\in \partial D$, see \cite[p. 39]{Ky}. 
Therefore, by the jump formula for the Martinelli-Bochner integral (for 
the double layer potentials), we have 
$$
 M^-_{\partial D} w_0 - M^+_{\partial D}w_0 =w_0 \mbox{ on } \partial D,
$$ 
$$
\overline \partial_\nu ( M^-_{\partial D} w_0 ) 
- \overline \partial_\nu(M^+_{\partial D} w_0) =0 \mbox{ on } \partial D,
$$ 
see \cite[Corollary 4.9]{Ky} for smooth domains and \cite{Cost88} for Lipschitz domains. 

In particular, as $M^{+}_{\partial D} w_0\equiv 
 0$ in ${\mathbb C}^n \setminus 
\overline D$, we obtain  
$$
{\rm t}^- M^{-}_{\partial D} w_0 =w_0  \mbox{ on }  \partial D, \,\,
\overline \partial_\nu ( M^-_{\partial D} w_0 ) ^- =0 \mbox{ on }  \partial D.
$$ 
By the very construction, the Bocnher-Martinelli integral $M^{-}_{\partial D}w_0$ 
is harmonic in $D$. Moreover \eqref{eq.M-.cont} yields 
$M^{-}_{\partial D} w_0 \in H^1  (D)$ and, thus, the function 
$w=M^{-}_{\partial D}w_0$ is a $H^1(D)$-solution to the Cauchy problem
$$
\left\{
\begin{array}{lll}
\Delta_{2n} w=  0 \mbox{ in }  D, \\ 
w =w_0 \mbox{ on }  \partial D, \\
\overline \partial_\nu w =0 \mbox{ on }  \partial D.
\end{array}
\right.
$$

As $D$ is a Lipschitz domain, then there is a sequence $\{ w_k\}\subset C^1 (\overline D)$ 
approximating $w$ in $H^1 (D)$. Then, integrating by parts, we obtain 
$$
\sum_{j=1}^n \|\overline \partial_j w \|^2 _{L^2 (D)}= \lim_{k \to +\infty}
\sum_{j=1}^n (\overline \partial_j w,\overline \partial_j w_k ) _{L^2 (D)} = 
\lim_{k \to +\infty} 
\int_{\partial D} \overline w_k \overline \partial_\nu w ds (y) =0
$$
because $\overline \partial_\nu w=0$ on $\partial D$. 
 Therefore, $ \overline \partial w =0$ weakly in $D$, and, then, by \cite[\S 24.7]{Vla},
it is holomorphic in $D$, i.e. $w \in H^1 (D) \cap {\mathcal O} (D)$
satisfies ${\rm t}^- w= w_0$ on  $\partial D$, that was to be proved.
\end{proof}

It follows from Lemma \ref{l.Sigma} immediately that 
conditions (2) and (3) of Theorem \ref{t.Sigma} are equivalent. 

Moreover, using Stokes formula, we see that (2) implies (1):
$$
\int_{\partial D} w_0 \, \overline \partial \psi 
= \int_{\partial D} {\rm t}^-w \, \overline \partial \psi
=\int_{D} \Big( \overline \partial  w \wedge \overline \partial \psi + 
 w \, \overline \partial (\overline \partial  \psi) \Big) 
=0
$$ 
for all $(n,n-2)$-differential forms $\psi$ with the coefficients of the class 
$C^1 (\overline D)$. 

Finally, let $w_0\in H^{1/2} (\partial D)$ 
be a $CR$-function on $\partial D$. As $D$ is bounded, there is a ball $B(0,R)$ 
containing $\overline D$. Then for $z \not \in \overline B(0,R)$ the Bochner-Martinelli 
kernel satisfies 
$$
\overline \partial _\zeta {\mathfrak U} (\zeta, z) = 0 \mbox{ in } \overline B(0,R).
$$ 
As the Dolbeault complex is exact on the smooth forms over convex domains, for any 
$z \not \in \overline B(0,R)$ there is a $(n,n-2)$-form $\psi_z (\zeta)$ with 
smooth coefficients in  $\overline B(0,R)$, 
satisfying 
$$
\overline \partial _\zeta \psi_z (\zeta) = 
{\mathfrak U} (\zeta, z)  \mbox{ in } B(0,R),
$$
see \cite{KohnJJ}. Therefore, according to \eqref{eq.CR.weak}, 
$$
(M_{\partial D} w_0)(z) = \int_{\partial D} 
w_0 (\zeta) \, \overline \partial _\zeta \psi_z (\zeta) =0
\mbox{ for all } z \not \in \overline B(0,R). 
$$
In particular, as $(M_{\partial D} w_0)(z)$ is harmonic outside $\overline D$, 
the uniqueness theorem for harmonic functions yields $M^{+}_{\partial D}  w_0 \equiv 
 0$ in ${\mathbb C}^n \setminus \overline D$. Thus, (1) is equivalent to (2) and (3) (cf. 
also \cite[Theorem 7.1]{Ky} for domains with smooth boundaries), that was to be proved.
\end{proof}

\begin{corollary} \label{c.Sigma}
Let $D\subset {\mathbb C}^n$  be a bounded Lipschitz domain.  
Then the following conditions are equivalent:
\begin{itemize}
\item[(1)]
 the element $v\in {\mathcal H} (\hat {\mathbb C}^{n} \setminus D) $ belong to 
 the space $\Sigma (\hat {\mathbb C}^{n} \setminus D) $; 
\item[(2)]
there is a function $w \in H^1 (D) \cap {\mathcal O} (D)$ satisfying ${\rm t}^-(w)= 
{\rm t}^+(v) $ on $\partial D$; 
\item[(3)]
$M^{+}_{\partial D} {\rm t}^+ v \equiv 
 0$ in ${\mathbb C}^n \setminus \overline D$.
\end{itemize}
\end{corollary}

Next, denote by $({\mathcal O} (D))^*$ the dual space of ${\mathcal O} (D)$, i.e., the
space of all continuous linear functionals on ${\mathcal O} (D)$.
As usual, we give $({\mathcal O} (D))^*$ the strong topology, i.e.,
   the topology of the uniform convergence of functionals on bounded subsets of
   ${\mathcal O} (D)$, see \cite[Ch IV, \S 6]{Shaef}.
	
Then for $n=1$ the classical results \cite{Grot2},   \cite{Koth2} , \cite{Silva} for a simply connected bounded 
plane domain $D$, state that duality \eqref{eq.dual.compl} holds true and the 
related pairing 
\begin{equation} \label{eq.pair.CR.compl.map}
\langle \cdot, \cdot \rangle_{1}: {\mathcal O} (D) \times {\mathcal O}  ( \hat {\mathbb C} 
\setminus  D) \to 
\mathbb R
\end{equation} 
is given by the the curvilinear integral 
\begin{equation} \label{eq.pair.CR.compl}
\langle u, h \rangle _{1}= \int_{\partial G} v (z)  u(z) dz 
\end{equation} 
for $ u \in {\mathcal O} (D)$ and $h \in {\mathcal O}(\hat {\mathbb C} \setminus D)$ where 
$G\Subset D$ and a piece-wise smooth curve $\partial G$ belongs to 
the intersection of the domains of $u$ and $h$, respectively (of course, the pairing 
does not depend on $G$ with the above prescribed properties).

To employ the machinery of the general theory of partial differential equations, 
denote by $S_{\overline \partial ^*} (U)$ 
the space of smooth solutions to the equation $\overline \partial ^* g  =0 $ on an open 
set $U \subset {\mathbb C}^n$; actually, these solutions are $n$-rows $g= (g_1, \dots, g_n)$ 
of functions $g_j \in C^\infty (U)$ satisfying 
$$
\sum_{j=1}^n \overline \partial_j^* g_j =0 \mbox{ in } U. 
$$
If $n=1$ then the space $S_{\overline \partial ^*}  (U)$ is just the space of 
anti-holomorphic functions and  the  space  $S_{\overline \partial ^*}  (\hat {\mathbb C} 
\setminus D)$ has similar properties as the space
${\mathcal O}(\hat {\mathbb C} \setminus D)$. Clearly,
the complex  conjugation induces a topological anti-linear isomorphism
$$
{\mathcal O} (\hat {\mathbb C} \setminus D) \ni h  \to \overline h =g\in 
S_{\overline \partial ^*}  (\hat {\mathbb C} \setminus D).
$$ 
In particular, in this case the pairing 
\begin{equation} \label{eq.pair.CR.adj.compl}
\langle u, g \rangle _{2}= \int_{\partial G} \overline{g} (z)  \, u(z) \, dz 
\end{equation}
defined for $ u \in {\mathcal O} (D)$ and $g \in S_{\overline \partial ^*}  
(\hat {\mathbb C} \setminus D)$ induces the topological (anti-linear) isomorphism, see 
\cite{Grot1}, 
\begin{equation} \label{eq.dual.adj.compl}
({\mathcal O} (D))^* \cong S_{\overline \partial ^*} (\hat {\mathbb C} \setminus D).
\end{equation}

 A. Grothendieck proved that 
for any elliptic operator $A$ on a smooth manifold $X$ admitting a regular bilateral 
fundamental solution a tolological isomorphism
\begin{equation} \label{eq.dual.A.compl.over.G.l=k}
(S_A(D))^* \cong S_{A^*} ( \hat X\setminus  D)
\end{equation}
holds true, where $\hat X$ indicates that one should consider solutions regular 
``at the infinity'' with respect to the chosen fundamental solution, 
see  \cite[Theorem 4]{Grot1} for details.

Unfortunately, for $n>1$ the Cauchy-Riemann  operator $\overline \partial $ does not admit 
 a bilateral fundamental solution because of the Hartogs theorem on the removability 
of compact singularities and so, \cite[Theorem 4]{Grot1} is not applicable in this case. 
Nevertheless, the operator $\overline \partial $ 
admits regular left fundamental solutions, for instance, one of them is presented 
by the Bochner-Martinell kernel ${\mathfrak U}_n (\zeta,z)$, $n>1$. Then 
\cite[Theorem 3]{Grot1} suggests us that the pairing 
\begin{equation} \label{eq.pair.CRn.compl}
\langle u, g \rangle _{3}= \int_{\partial G} 
\sum_{j=1}^n \overline g_j (z) \, u(z) \, (-1)^{j-1}\, d\overline z[j] \wedge dz 
\end{equation}  
defined for $ u \in {\mathcal O} (D)$ and $g \in S_{\overline \partial ^*} 
(\hat {\mathbb C}^n \setminus D)$, where 
$G\Subset D$ and a piece-wise smooth surface $\partial G$ belongs to 
the intersection of the domains of $u$ and $g$ respectively,
induces a continuous surjective mapping 
\begin{equation} \label{eq.dual.A.compl.over.G.l>k}
S_{\overline \partial ^*} 
(\hat {\mathbb C}^n \setminus D) \to ({\mathcal O} (D))^* 
\end{equation}
(again, the pairing does not depend on $G$ with the above prescribed properties).

But the operator $\overline \partial ^*$ is not elliptic for $n>1$ and therefore 
the space $S_{\overline \partial ^*} (\hat {\mathbb C}^n \setminus D) $ is 
too large for mapping \eqref{eq.dual.A.compl.over.G.l>k} to be bijective, see example below.

\begin{example}
Let $n\geq 1$ and let $D=B (0,1)$ be the unit ball centered at the origin. 
In order to show that mapping \eqref{eq.dual.A.compl.over.G.l>k} is not injective for $n>1$ 
we employ the harmonic homogeneous functions. 

Namely, 
 let $\{ h^{(j)}_{r} \}$, $r\geq 0$, be the orthonormal basis in the Lebesgue space 
$L^2 (\partial B (0,1))$ over the sphere $\partial B (0,1)$ in ${\mathbb R}^{2n}$, $n\geq 1$, 
consisting of spherical harmonics of degree $r$, see, for instance, \cite[Ch. XI]{So}, where 
$j$ is the number of the polynomial in the basis, $1\leq j\leq J(r,2n)$, 
$J_{0,2}=1$, $J_{r,2}=2$, $r \in \mathbb N$, and 
$J_{r,2n}=\frac{(2n+2r-2)\, (r+2n-3)!}{r! (2n-2)!}$, $n>1$. 
For each $r$, the harmonic continuation of $h_r$
to $D$ gives the harmonic homogeneous polynomial that we will still
 denote by $h_r$.  The harmonic continuation of $h_r$
to ${\mathbb R}^{2n} \setminus \overline D$ is
given by   
\begin{equation}\label{eq.harm.ext}
\frac{h_r (x)}{|x|^{2n+2r-2}}.
\end{equation}

If $n\geq 1$ then any vector function 
$ \overline \partial \Big(\frac{ h_r (x)}{|x|^{2n+2r-2}}\Big)$, $r \in {\mathbb Z}_+$,   
satisfies  
\begin{equation}  \label{eq.harmonic.hom}
\overline \partial  ^* \Big(
\overline \partial \Big(\frac{ h_r (x)}{|x|^{2n+2r-2}} \Big)\Big) 
=0 \mbox{ in } {\mathbb C}^{n}\setminus 
\{0\} ,
\end{equation}
and hence $\overline \partial \frac{ h_r (x)}{|x|^{2n+2r-2}}$ belongs 
to the space $S_{\overline \partial ^*} (\hat {\mathbb C}^{n} \setminus D)$ for all 
$r \in {\mathbb Z}_+$. 

Using the complex structure 
one may rewrite the harmonic homogeneous polynomials $h_r$ as 
$$
h_{p,q} (z,\overline z) = \sum_{|p+q|=r} a_{p,q} z^p \overline z^q
$$ 
where $p=(p_1, \dots p_n),  q=(p_1, \dots q_n) \in {\mathbb Z}^{n}_+$, 
$z^p = z_1^{p_1} \dots z_n^{p_n}$, $\overline z^p = 
\overline z_1^{q_1} \dots \overline z_n^{q_n}$, 
and $a_{p,q}$ are suitable complex coefficients, providing the harmonicity,  
see \cite[Ch. 1, \S 5]{Ky}. 

Now formula \eqref{eq.harmonic.hom} implies that the space 
$S_{\overline \partial^*} (\hat {\mathbb C}^{n} \setminus B(0,1))$  contains  the 
convergent in neighbourhoods of ${\mathbb C}^n \setminus B(0,1)$ series of the type   
\begin{equation}  \label{eq.harmonic.decomp}
\sum_{r=0}^\infty \sum_{|p+q|=r}  a^{(j)}_{p,q} \sum_{j=1}^{J(r)}
\overline \partial \Big(\frac{  z^p \overline z^q }{|z|^{2n+2r-2}} \Big)
\end{equation}  
with suitable complex coefficients $a^{(j)}_{p,q}$ and, moreover, it 
contains non-zero elements $g$
annihilating Grothendieck's pairing \eqref{eq.pair.CRn.compl} for all 
$u\in {\mathcal O}(B(0,1))$ if $n>1$. 

Indeed, take the vector 
$g^{(p,q)}= \overline \partial \Big( \frac{ 2\overline z^q z^p}{|z|^{2n+2|q|-2}} \Big) \in 
S_{\overline \partial^*} (\hat {\mathbb C}^{n} \setminus D)$. 

If $n=1$ then 
$$
\frac{ 2\overline z^q z^p}{|z|^{2n+2|q|-2}}= 
\frac{ 2\overline z^q z^p}{|z|^{2|q|}} = \frac{2 z^p }{z^q}
$$ 
and hence $g^{(p,q)}\equiv 0$ in ${\mathbb C} \setminus \{0\}$ in this case. 
Thus, we proceed with $n>1$. 

As it is known, for any smooth domain ${\mathcal D} \subset {\mathbb R}^n$ we have
$$
(-1)^{j-1}  dx[j] =\nu_j (x) ds (x) \mbox{ on } \partial {\mathcal D} ,
$$
where $\nu (x)= (\nu _1(x), \dots \nu _n(x))$ is the exterior unit normal vector to 
$\partial {\mathcal D}$ at the point 
$x$ and $ds $ is the volume form on $\partial D$. 
Clearly, for any  sphere $S_R$ centered at the origin with a radius $R$ 
we have $ \nu(x) = x/R$. Therefore 
$$
(-1)^{j-1} d\overline z[j] \wedge dz = 2^{n-1}\iota^n 
\frac{ z_j}{2R} \, ds (z,\overline z) 
$$
on $S$, see \cite[Lemma 3.5]{Ky}. 

Then, using Euler formula for positively 
homogeneous functions, we see that  
$$
\sum_{j=1}^n \overline g^{(p,q)}_j (z) \,  (-1)^{j-1}\, d\overline z[j] \wedge dz =
2^{n-1}\iota^n  \sum_{j=1}^n \frac{z_j}{R} 
 \overline{\overline \partial _j 
\Big( \frac{ \overline z^q z^p}{|z|^{2n+2|q|+2|p|-2}} \Big) }
\, ds   =
$$
\begin{equation} \label{eq.form}
\frac{2^{n-1}\iota^n}{R}  \overline{\sum_{j=1}^n 
 \overline z_j \overline \partial _j 
\Big( \frac{ \overline z^q z^p}{|z|^{2n+2|q|+2|p|-2}} \Big) }
\, ds = 2^{n-1}\iota^n  (1-|p|-n)  \frac{z^{q} \overline z^p ds }{R^{2n+2|q|+2|p|-1}}
\end{equation}
on the sphere  $S_R$.

Finally, as Gothendieck's pairing $\langle \cdot ,\cdot \rangle_{3} $ 
does not depend on the choice of the domain $G$ we may take 
the ball $G=B(0,R)$  of radius $R$ centered at the origin. Then 
$S_R \subset B(0,1)$ and formula \eqref{eq.form} yields for 
all $q\in {\mathbb Z}_+$ with $|q|\geq 1$:
$$
\langle z^s ,g^{(0,q)}\rangle_{3} = 2^{n-1}\iota^n 
(1-n) \int_{S_R} \frac{z^{q+s} ds }{R^{2n+2|q|-1}} =0 \mbox{ for all }
s \in {\mathbb Z}^n_+ 
$$
by the famous property of holomorphic functions. Therefore, 
$$
\langle u, g^{(0,q)} \rangle_{3} =0 \mbox{ for all } u\in {\mathcal O}(B(0,1)),
$$ 
for infinitely many elements $g^{(0,q)}$ with $|q|\geq 1$ that was to be proved.

At the end  we note that, according to formula
\eqref{eq.harm.ext} and Corollary \ref{c.Sigma},  the space 
$\Sigma (\hat {\mathbb C}^{n} \setminus B(0,1))$  coincides with set of all  the 
convergent in neighbourhoods of ${\mathbb C}^n \setminus B(0,1)$ series of the type   
\begin{equation*} 
 \sum_{|p|\geq 0}  \frac{a^{(j)}_{p}  z^p  }{|z|^{2n+2|p|-2}}
\end{equation*}  
with suitable complex coefficients $a^{(j)}_{p}$. Of course, the vector 
$$
g^{(p)}(z)=\overline \partial \Big( 
 \frac{  z^p  }{|z|^{2n+2|p|-2}}\Big)  
$$
belongs to the space $S_{\overline \partial^*} (\hat {\mathbb C}^{n} \setminus D)$. However, 
according to \eqref{eq.form},
$$
\langle z^p, g ^{(p)}\rangle_{3} =
2^{n-1}\iota^n 
(1-|p|-n) \int_{S_R} \frac{|z|^{2p_1} \dots |z|^{2p_n} ds }{R^{2n+2|p|-1}} \ne 0 
\mbox{ for all } 
p \in {\mathbb Z}^n_+ ,
$$
i.e. $g^{(p)}$ does not annihilate the space ${\mathcal O} (B(0,1))$.
\end{example}

\section{The proof of Theorem \ref{t.dual.compl.over.G}}
\label{s.proof}

Actually, the proof goes alongside with the classical scheme for the dualities  
\eqref{eq.dual.compl} and \eqref{eq.dual.A.compl.over.G.l=k}.

\subsection{The pairing and the mapping.} 
First, we recall that the elements of the space ${\mathcal H} (U)$
are actually infinitely differentiable for any open set $U \subset {\mathbb R}^n$ 
(they are even real analytic at each point $y \in U$). In particular, for a closed 
set $\sigma \subset {\mathbb C}^n$,  
any element of the space ${\mathcal H} (\sigma)$
is actually infinitely differentiable  on an open set $U \Supset \sigma$. 

Pick functions $u \in {\mathcal O} (D) $ and $v \in \Sigma ( \hat {\mathbb C}^n \setminus  
D)$. By the discussion above, there is a (unbounded) domain $U_v$ containing 
${\mathbb C}^n \setminus D$ and such that $v \in {\mathcal H} (U_v) \cap C^\infty (U_v) $. 
Hence $V_v=U_v \cap D$ is an open set in ${\mathbb C}^n$. Moreover, as $\partial D$ is 
connected, there is a smooth closed surface $\Gamma \subset V_v$ 
that is a boundary of a bounded domain $G\Subset D$. 

Then we may define the pairing $\langle \cdot, \cdot \rangle $ between the spaces 
${\mathcal O}(D)$ and $\Sigma (\hat {\mathbb C}^n \setminus D)$: 
\begin{equation} \label{eq.pair.A.compl.NShT}
\langle u, v \rangle = \int_{\Gamma}  \sum_{j=1}^n (-1)^{j-1}
\overline{(\overline \partial _j v_j) (z)} \,  u(z) \,  d\overline z[j] \wedge dz  \Big)  
\end{equation} 
for $ u \in {\mathcal O} (D)$ and $v \in \Sigma ( \hat {\mathbb C}^n \setminus  D)$. 
Note that the vector $w= \overline \partial  v$ belongs to $S_{
\overline \partial  ^*}( \hat {\mathbb C}^n \setminus  D)$ 
if $v \in \Sigma ( \hat {\mathbb C}^n \setminus  D)$ and hence 
pairing \eqref{eq.pair.A.compl.NShT} is still closely related to the Grothendieck pairing 
\eqref{eq.pair.CRn.compl}. 

Next, according to Stokes' formula, for any two surfaces 
$\Gamma_1$, $\Gamma_2$ 
with the declared above properties we have 
$$
\int_{\Gamma_1}  \sum_{j=1}^n (-1)^{j-1}
\overline{(\overline \partial _j v_j) (z)} \,  u(z) \,  d\overline z[j] \wedge dz  \Big) - 
\int_{\Gamma_2}  \sum_{j=1}^n (-1)^{j-1}
\overline{(\overline \partial _j v_j) (z)} \,  u(z) \,  d\overline z[j] \wedge dz  \Big) = 
$$
\begin{equation*} 
\int_{\Omega} \Big( 
\sum_{j=1}^n\overline{(\overline \partial_j v)} \, \overline \partial_j u 
 - \frac{1}{4} \Delta_{2n} u  \Big) d\overline z\wedge dz  =0
\end{equation*} 
for all $ u \in {\mathcal O}(D)$ and all $v \in\Sigma (\hat {\mathbb C} ^n\setminus D)$, 
where $\Omega$ is the open set bounded by the surfaces $\Gamma_1$, $\Gamma_2$. 
Thus, pairing \eqref{eq.pair.A.compl.NShT} does not depend on a particular choice of 
$\Gamma \subset V_v$.

Obviously, 
\begin{equation} \label{eq.pair.A.compl.est}
|\langle u, v \rangle| \leq  C_{\Gamma} \max_{z\in \Gamma} 
|\nabla v (z)|   \max_{x\in \Gamma} |u (z)| 
\end{equation}
with a constant $C_{\Gamma}$ independent on $ u \in {\mathcal O}(D)$ and all 
$v \in \Sigma (\hat {\mathbb C}^n \setminus D)$. Therefore, taking in account the topologies 
of the spaces under the consideration, pairing \eqref{eq.pair.A.compl.NShT}
induces a sesquilinear separately continuous mapping 
\begin{equation} \label{eq.pair.A.compl.map}
\langle \cdot, \cdot \rangle: {\mathcal O}(D) \times  
\Sigma (\hat {\mathbb C}^n \setminus D) \to 
\mathbb C
\end{equation}
In particular, for any fixed $v \in \Sigma ( \hat{\mathbb C}^n \setminus D)$ the functional 
\begin{equation} \label{eq.fv}
f_v (u) = \langle u, v \rangle, \,\, u  \in {\mathcal O}(D), 
\end{equation}
is bounded and linear, i.e. $f_v  \in ({\mathcal O}(D))^*$. 
Moreover, by \eqref{eq.pair.A.compl.est}, the mapping 
\begin{equation} \label{eq.mapping.compl}
\Sigma ( \hat {\mathbb C}^n \setminus  D) \ni v \to f_v \in ({\mathcal O}(D))^*
\end{equation}
is anti-linear and continuous. 

\subsection{The injectivity of the mapping.}  
Let us prove that  mapping \eqref{eq.mapping.compl} is injective. Indeed, let 
\begin{equation*}
\langle u, v \rangle = 0 \mbox{ for all } u  \in {\mathcal O}(D). 
\end{equation*}
According to Corollary \ref{c.Sigma}, 
there is a function $w \in H^1 (D) \cap {\mathcal O} (D)$ satisfying ${\rm t}^-(w)= v$ 
on $\partial D$.  In particular, 
\begin{equation*}
\langle w , v\rangle = 0 . 
\end{equation*}
On the other hand, as  $w \in H^1 (D) \cap  {\mathcal O} (D)$, 
then Stokes' formula  yields
\begin{equation} \label{eq.limit.1}
0=\langle w, v \rangle = \int_{\Gamma}  \sum_{j=1}^n (-1)^{j-1}
\overline{(\overline \partial _j v_j) (z)} \,  w(z) \,  d\overline z[j] \wedge dz  \Big) = 
\end{equation} 
\begin{equation*} 
\int_{\partial D} 
  \sum_{j=1}^n (-1)^{j-1}
\overline{(\overline \partial _j v_j) (z)} \,  v(z) \,  d\overline z[j] \wedge dz  \Big) =
\end{equation*}
\begin{equation*}
\lim_{R\to + \infty}\Big( 
\int_{B (0,R) \setminus \overline  D} 
|\overline \partial v (z)|^2 dx - \int_{|z|=R}\sum_{j=1}^n (-1)^{j-1}
\overline{(\overline \partial _j v_j) (z)} \,  v(z) \,  d\overline z[j] \wedge dz  \Big). 
\end{equation*}
Since $v $ is harmonic in ${\mathbb R}^{2n} \setminus D$ and vanishes at the infinity, 
we have  
$$
|\partial^\alpha v(z)| \leq c_1 |z|^{2-2n-|\alpha|}, \mbox{ if } n\geq 2, 
\alpha \in {\mathbb Z}^{2n}_{+}, 
$$ 
see, for instance, \cite[\S 24.10, formulae (33)-(35)]{Vla}. In particular, 
\begin{equation} \label{eq.limit.2}
\lim_{R\to + \infty}\Big( 
\int_{|z|=R}\sum_{j=1}^n (-1)^{j-1}
\overline{(\overline \partial _j v_j) (z)} \,  v(z) \,  d\overline z[j] \wedge dz  \Big)=0. 
\end{equation}

Therefore \eqref{eq.limit.1}, \eqref{eq.limit.2} imply
\begin{equation*}
0=\langle w, v \rangle  = \sum_{j=1}^n \|\overline 
\partial_j v\|^2_{L^2({\mathbb C}^n \setminus \overline  D)},
\end{equation*}
i.e. $v \in {\mathcal O} ( \hat{\mathbb C}^n \setminus \overline D) $. Moreover 
it belongs to $ {\mathcal O} ( \hat{\mathbb C}^n \setminus D) $ because it is 
harmonic in a neighbourhood of $\hat{\mathbb C}^n \setminus D$. 
As  $n>1$, Hartogs Theorem on the removable of compact singularities 
for holomorphic functions of several variables immediately yields
that $v$ extends as holomorphic functions over all ${\mathbb C}^n$. 

In particular, as $v \in {\mathcal O}({\mathbb C}^n)$ 
vanishes at the infinity, it is identically zero because 
of the Liouville Theorem, i.e. mapping \eqref{eq.mapping.compl} is anti-linear  continuous 
and injective.  

\subsection{The surjectivity of the mapping.}  
Let us prove that  \eqref{eq.mapping.compl} is surjective. 

Indeed, fix an element $f \in ({\mathcal O} (D))^*$. 
As ${\mathcal O} (D)$ is a closed subspace of $C(D)$, by the Khan-Banach 
Theorem there is a functional $F \in (C(D))^*$ coinciding 
with $f$ on $ {\mathcal O}(D)$. Then, by the classical Riesz duality for the space  
$C(D)$, there is a compactly supported in $D$ measure $\mu$,  such that 
\begin{equation} \label{eq.Riesz}
f(u) = \int_{K}  u (z) d\mu (z) \mbox{ for all } u \in {\mathcal O} (D),
\end{equation}
where the compact $K\Subset D$ contains the support ${\rm supp} (\mu)$  of $\mu$, 
see, for instance, \cite[\S 4.10]{Edw}.

Since $K\Subset D$  there is a domain $G$ with smooth boundary
 such that  $K\Subset G\Subset D$. In particular, by the Bochner-Martinelli formula 
\eqref{eq.Green.2} , we have 
\begin{equation} \label{eq.Green.2G}
(M_{\partial G}  u) (z) 
=  \left\{
\begin{array}{ll} u(z) & z \in G, \\
0, & z \not \in \overline G, \\
\end{array}
\right.
\end{equation} 
for any $u \in  {\mathcal O} (D)$. 

Now, formulae \eqref{eq.MB.*}, \eqref{eq.Riesz}, \eqref{eq.Green.2G} and Fubini theorem yield
\begin{equation} \label{eq.surj}
f(u) = \int_{K} 
\Big(\int_{\partial G} {\mathfrak U}_n (\zeta,z) u(\zeta)   \Big)  d\mu (z) =
\end{equation}
\begin{equation*}
\left\{ 
\begin{array}{lll}
\int_{\partial G} h (\zeta) \,  u(\zeta) \,  d\zeta , & n=1, 
\\[.2cm]
\int_{\partial G} \sum_{j=1}^n (-1)^{j-1}
\overline{(\overline \partial _j \hat v) (\zeta)} \,  u(\zeta) \,  d\overline \zeta[j] 
\wedge d\zeta  , & n>1, 
\\
\end{array}
\right.
\end{equation*}
for all  $u \in {\mathcal O} (D)$ where  
$$
h (\zeta) = \frac{1}{2\pi \iota } \int_{K} \frac{d \mu (z)}{\zeta-  z} , \, \, 
\hat v(\zeta)  =  \int_{K}  \Phi_{2n} (\zeta,z)  d\overline \mu (z).
$$

If $n=1$ then the function $h$ is holomorphic 
in $ {\mathbb C} \setminus K$ and vanishes at the infinity because of the  
behaviour of the Cauchy kernel, i.e. $h \in {\mathcal O}  
(\hat {\mathbb C} \setminus D)$ because $K\Subset D$. This gives us  classical duality 
\eqref{eq.dual.compl} by A.Grothendieck \cite{Grot2}, G. K{\"o}the   \cite{Koth2},  and  J. 
Sebasti\~{a}o e Silva \cite{Silva} related to pairing \eqref{eq.pair.CR.compl}. Then the 
replacement $g=\overline h$ induces classical duality \eqref{eq.pair.CRn.compl} 
by  A.Grothendieck  \cite{Grot1} related to pairing \eqref{eq.pair.CR.adj.compl}. 

For $n>1$, as the kernel $\Phi_{2n} (\zeta,z)$ represents the right fundamental solution to 
the Laplace operator
with respect to the variable $\zeta$, we see that $\hat v $ is harmonic 
in $ {\mathbb C}^n \setminus K$ and vanishes at the infinity. 
In particular, $\hat v \in {\mathcal H}  (\hat {\mathbb C}^n \setminus D)$  
because $K\Subset D$.  

However we can not grant easily that $\hat v \in \Sigma (\hat {\mathbb C}^n 
\setminus D)$.   

To resolve this difficulty, let us fix $u\in {\mathcal O}(D)$ and a sequence 
$\{ u_\nu \} \subset {\mathcal O} (D) \cap H^1 (D)$ approximating $u$ in  
$C (D)$; it exists by the hypothesis of the theorem. Then, by  \eqref{eq.surj} 
and Stokes' formula, 
\begin{equation} \label{eq.limit} 
f(u)  = \lim_{\nu  \to \infty} f(u_\nu) = 
\end{equation} 
\begin{equation*}
\lim_{\nu  \to \infty}
\int_{\partial G} 
\sum_{j=1}^n (-1)^{j-1}
\overline{(\overline \partial _j \hat v) (\zeta)} \,  u_\nu(\zeta) \,  d\overline \zeta[j] 
\wedge d\zeta = 
\end{equation*}
\begin{equation*}
\lim_{\nu  \to \infty}
\int_{\partial D} \sum_{j=1}^n (-1)^{j-1}
\overline{(\overline \partial _j \hat v) (\zeta)} \,  u_\nu(\zeta) \,  d\overline \zeta[j] 
\wedge d\zeta. 
\end{equation*} 
On the other hand, we recall the following 
result of \cite{Rom}. For a function $w_0\in H^{1/2} (\partial D)$ 
 denote by ${\mathcal P}_{ D} (w_0)$ 
the unique solution $w \in H^1 (D)$ to the interior 
Dirichlet problem for the Laplace equation:
 \begin{equation} \label{eq.Dir.int}
\left\{ 
\begin{array}{lll}
\Delta_{2n} {\mathcal P}_{D} (w_0) =0 & \mbox{ in } &  D, \\
{\mathcal P}_{ D} (w_0) = w_0 & \mbox{ on }  & \partial D. \\ 
\end{array}
\right.
\end{equation}
Similarly,  denote by $\tilde  {\mathcal P}_{ D} (w_0) $ 
the unique solution to the exterior Dirichlet problem for the Laplace equation:
 \begin{equation} \label{eq.Dir.ext}
\left\{ 
\begin{array}{lll}
\Delta_{2n} \tilde  {\mathcal P}_{ D} (w_0) =0 & \mbox{ in } & {\mathbb C}^n \setminus 
\overline D, \\
\tilde {\mathcal P}_{D} (w_0) = w_0 & \mbox{ on }  & \partial D, \\ 
\end{array}
\right.
\end{equation}
satisfying \eqref{eq.zero.infty} and such that 
$\overline \partial 
\tilde{\mathcal P}_{ D} (w_0) \in L^2 ({\mathbb C}^n \setminus \overline D)$. 

\begin{lemma} \label{l.NaSh}
The Hermitian form
$$
h_{D} (w, \tilde w) = 
 \sum_{j=1}^n \int_{D}
\overline{(\overline \partial_j w)} \, \overline \partial_j \tilde w  \, dx 
+ \sum_{j=1}^n \int_{{\mathbb C}^n \setminus \overline D} 
 \overline{( \overline \partial_j \tilde  {\mathcal P}_{ D} (w))} \, \overline \partial_j 
\tilde {\mathcal P}_{ D} (\tilde w) \, dx
$$
defines an inner product on $H^1 (D)$ and the topologies induced in $H^1 (D)$ by
$h_{D} (w, \tilde w)$ and by the standard inner product are equivalent.
\end{lemma}

\begin{proof} See \cite{Rom} (or \cite{NaSh} for more advanced properties).
\end{proof}

Then Stokes' formula implies that 
$$
- \int_{\partial D} \sum_{j=1}^n (-1)^{j-1}
\overline{(\overline \partial _j \hat v) (\zeta)} \,  h(\zeta) \,  d\overline \zeta[j] 
\wedge d\zeta = h_D (h,{\mathcal P}_D \hat v)  
$$
for each $h \in H^1 (D) \cap {\mathcal O} (D)$. In particular, \eqref{eq.limit}  yields
\begin{equation} \label{eq.f.limit}
f(u)  = -  \lim_{\nu  \to \infty}  h_{D} (u_\nu,{\mathcal P}_D \hat v) = -
\lim_{\nu  \to \infty}  h_{D} (u_\nu, \Pi_D {\mathcal P}_D \hat v)   
\end{equation}
where $\Pi_D: H^1 (D) \to H^1 (D) \cap {\mathcal O}  (D)$ is the orthogonal projection 
related to the Hermitian form $ h_{D} (\cdot , \cdot)$. 

We are going to show that the function $\tilde{\mathcal P}_{ D}( \Pi_D {\mathcal P}_D w)$ 
belongs to the space ${\mathcal H} (\hat {\mathbb C}^n \setminus  D)$ if 
$w \in {\mathcal H} (\hat {\mathbb C}^n \setminus  D)$. Indeed, fix $w \in {\mathcal H} (\hat 
{\mathbb C}^n \setminus  D)$. By the definition there is a domain $G\Subset D$ with 
Lipschitz  $w \in {\mathcal H}(\hat {\mathbb C}^n \setminus  G)$. 
As the pairing does not depend on $G$, we choose it in such a way that 
the set $D\setminus G$ has no compact components in $D$.  
We endow the spaces $H^1 (D) $, $H^1 (G) $ 
with the inner products $h_{D} (\cdot , \cdot)$, $h_{G} (\cdot , \cdot)$, 
respectively.

Next we note that the exterior Dirichlet problem is well-posed and then the Hermitian form 
$$
\tilde h_{D} (W, \tilde W)  = \sum_{j=1}^n
\int_{D} \overline{(\overline \partial_j {\mathcal P}_D W)}  
\overline \partial_j  {\mathcal P}_D (\tilde W)  dx + 
\sum_{j=1}^n \int_{{\mathbb C}^n \setminus 
\overline D} \overline{( \overline \partial_j W)} \overline \partial_j  \tilde W  dx
 $$
is an inner product on the space $\tilde H^1 ({\mathbb C}^n \setminus 
\overline D) \cap  {\mathcal H}(\hat {\mathbb C}^n \setminus 
\overline D)$ consisting of such elements from ${\mathcal H} (\hat {\mathbb C}^n \setminus 
\overline D) \cap H^1_{\rm loc } ({\mathbb C}^n \setminus  D)$ that $\overline \partial_j W 
\in L^2 ({\mathbb C}^n \setminus \overline D)$ for all $1\leq j \leq n$. Moreover, by the 
construction, \begin{equation} \label{eq.inner.rel.D}
h_{D} ({\mathcal P}_D W, h) = 
\tilde h_{{\mathbb C}^n \setminus 
\overline D} (W, \tilde {\mathcal P}_{ D} h ) 
\end{equation}
for all $h \in H^1 (D) \cap {\mathcal H} (D) $ and $W\in \tilde H^1 ({\mathbb C}^n \setminus 
\overline D) \cap  {\mathcal H} (\hat {\mathbb C}^n \setminus 
\overline D)$. In other words, the mapping $\tilde {\mathcal P}_{ D}$ defines 
an isomorphism between the Banach spaces $H^1 (D) \cap {\mathcal H} (D)$
and $\tilde H^1 ({\mathbb C}^n \setminus \overline D) \cap {\mathcal H} (
\hat {\mathbb C}^n \setminus \overline D )$, and, similarly  
\begin{equation} \label{eq.inner.rel.G}
h_{G} ({\mathcal P}_G W, h) = 
\tilde h_{ G} (W, \tilde {\mathcal P}_{G} h ) 
\end{equation}
for all $h \in H^1 (G) \cap {\mathcal H} (G)$ and $W\in \tilde H^1 ({\mathbb C}^n \setminus 
\overline G) \cap {\mathcal H} (\hat {\mathbb C}^n \setminus 
\overline G)$.

We also denote by $Y^1(G)$ the closure of 
$H^1 (D) \cap {\mathcal O}(D)$ in $H^1 (G) \cap {\mathcal O} (G)$.  
By the Stieljes-Vitali theorem, the embedding $R:  H^1 (D) \cap {\mathcal H} (D) \to 
 H^1 (G) \cap {\mathcal H} (G)$ is compact. Then Hilbert theorem 
on the spectrum of compact self-adjoint operator implies that  
there is an orthonormal (with respect to 
$h_{D} (\cdot , \cdot)$) basis $\{ b_m\}_{m \in \mathbb N}$ 
in the space $H^1 (D) \cap {\mathcal H} (D)$ that form an orthogonal system 
(with respect to $h_{G} (\cdot , \cdot)$) in $H^1 (G) \cap {\mathcal H} (G)$, satisfying 
$$
\Pi_DR^*R\Pi_D b_m = \lambda_m b_m.
$$
Moreover, as the set $D\setminus G$ has no compact components in $D$, Mergelyan Theorem 
implies that the space $H^1 (D) \cap{\mathcal H}(D) $ is everywhere dense in $H^1 (G) \cap 
{\mathcal H}(G) $, i.e. the system $\{R b_m \}_{m \in \mathbb N}$ is an 
orthogonal basis  in the space  $H^1 (G) \cap {\mathcal H}(G) $, see \cite{ShTaSMJ}. 
In particular, the eigenvalues $\lambda_m \ne 0$ correspond to the subspace 
$H^1 (D) \cap {\mathcal O}(D)$ and the system $\{b_m \}_{\lambda_m\ne 0}$ is an orthonormal 
basis in this space and the system $\{R b_m \}_{\lambda_m\ne 0}$ is an orthogonal 
basis  in the space $Y^1 (G)$. Then the projection $\Pi_D$ is 
given by 
$$
\Pi_D  h = \sum_{\lambda_m\ne 0} h_{D} (h, b_m) b_m \mbox{ for all } h \in H^1 (D) \cap 
{\mathcal H}(D).
$$
However, by \eqref{eq.inner.rel.D}, we may actually consider the operator $\Pi_D {\mathcal P}_D$ as the orthogonal 
projection from  $\tilde H^1 ({\mathbb R}^n \setminus 
\overline D) \cap {\mathcal H} (\hat {\mathbb C}^n \setminus 
\overline D)$ to the closed subspace of functions $W$ with $
\overline \partial P_D W=0$ in $D$, i.e.
$$
\tilde{\mathcal P}_{D}\Pi_D  h = \sum_{\lambda_m\ne 0} h_{D} (h, b_m) \tilde{\mathcal P}_{D} b
_m \mbox{ for all }  h \in H^1 (D) \cap 
{\mathcal H}(D).
$$
If we denote by $\tilde R$ the continuous embedding operator 
$\tilde R: {\mathcal H} (\hat {\mathbb C}^n \setminus  \overline G) \cap \tilde H^1  
 ({\mathbb C}^n \setminus  \overline G) \to  {\mathcal H} (\hat {\mathbb C}^n \setminus  \overline D) \cap \tilde H^1   ({\mathbb C}^n \setminus  \overline D)$ then 
for any $w \in {\mathcal H} (\hat {\mathbb C}^n \setminus  D)$ we have 
$$
 w= \sum_{m=1}^\infty h_{G} ({\mathcal P}_{G} w , R b_m)  \frac{\tilde{\mathcal P}_{G} Rb_m}
{h_G (Rb_m, Rb_m)}, \, 
\tilde R w= \sum_{m=1}^\infty h_{D} ({\mathcal P}_{D} w, b_m) \tilde{\mathcal P}_{D} b_m,  
$$
\begin{equation} \label{eq.Pi.1}
\tilde{\mathcal P}_{D} \Pi_D {\mathcal P}_D\tilde R w= \sum_{\lambda_m\ne 0} 
h_{D} ({\mathcal P}_{D} \tilde R w, b_m) \tilde{\mathcal P}_{D} b_m.  
\end{equation}
Of course, by the Uniqueness Theorem for harmonic functions, $w$ is the unique 
harmonic extension of $Rw$ from ${\mathbb C}^n \setminus  D$ to ${\mathbb C}^n \setminus  G$. 
Besides, if $\lambda_m \ne 0$ then 
$$
h_G (Rb_m, Rb_m) = h_D (\Pi_DR^*R\Pi_D b_m, b_m) = \lambda_m
$$
because $h_D (b_m, b_m) =1$ and then the series 
\begin{equation} \label{eq.Pi.2}
\sum_{\lambda_m\ne 0} h_{G} ({\mathcal P}_{G} w , R b_m)  \frac{\tilde{\mathcal P}_{G} Rb_m}
{h_G (Rb_m, Rb_m)} = \sum_{\lambda_m\ne 0} h_{G} ({\mathcal P}_{G} w , R b_m)  \frac{\tilde{\mathcal P}_{G} Rb_m}
{\lambda_m}  
\end{equation}
converges in  the space ${\mathcal H} (\hat {\mathbb C}^n \setminus  \overline G) \cap \tilde H^1  ({\mathbb C}^n \setminus  \overline G)$.

On the other hand, if $w \in {\mathcal H} (\hat {\mathbb C}^n \setminus  D)$, 
then, by Stokes' formula, 
\begin{equation} \label{eq.Pi}
h_{D} ( {\mathcal P}_D \tilde R w, \Pi_D h) 
= -\int_{\partial D}
\sum_{j=1}^n (-1)^{j-1}
\overline{(\overline \partial _j w) (\zeta)} \,  \Pi_D h(\zeta) \,  d\overline \zeta[j] 
\wedge d\zeta = 
\end{equation}
\begin{equation*} 
-\int_{\partial G}
\sum_{j=1}^n (-1)^{j-1}
\overline{(\overline \partial _j w) (\zeta)} \, \Pi_D h(\zeta) \,  d\overline \zeta[j] 
\wedge d\zeta = 
 h_{G} ({\mathcal P}_G w, R \Pi_D h)  . 
\end{equation*}
for all $h \in H^1 (D) \cap {\mathcal H} (D)$. 
In particular, 
\begin{equation} \label{eq.Pi.3} 
\tilde{\mathcal P}_{D} \Pi_D {\mathcal P}_D\tilde R w= \sum_{\lambda_m\ne 0} 
 h_{G} ({\mathcal P}_G w, R b_m) \tilde{\mathcal P}_{D} b_m.  
\end{equation}

Moreover, since ${\mathcal P}_G \tilde {\mathcal P}_G    =I$, 
formula \eqref{eq.Pi} implies 
$$
h_{D} ( {\mathcal P}_D \tilde R \tilde {\mathcal P}_G Rb_m, \Pi_D h) = 
 h_{G} (Rb_m, R R\Pi_D h) = \lambda_m 
h_{D} (b_m, \Pi_D h) , 
$$
for all $h \in H^1 (D) \cap {\mathcal H} (D)$ if $\lambda_m \ne 0$. Thus, 
as  $H^1 (D) \cap {\mathcal O}(D)$ 
is densely embedded to $Y^1 (G)$ and  the mapping $\tilde {\mathcal P}_{ G}$ defines 
an isomorphism between the Banach spaces $H^1 (G) \cap {\mathcal H} (G)$
and $\tilde H^1 ({\mathbb C}^n \setminus \overline G) \cap {\mathcal H} (
\hat {\mathbb C}^n \setminus \overline G )$, formulas \eqref{eq.Pi.1}, \eqref{eq.Pi.2}, 
\eqref{eq.Pi.3} yield 
$$
\tilde{\mathcal P}_{D} \Pi_D {\mathcal P}_D\tilde R w = 
\sum_{\lambda_m\ne 0} h_{G} ({\mathcal P}_{G} w , R b_m)  \frac{\tilde{\mathcal P}_{G} Rb_m}
{h_G (Rb_m, Rb_m)}.
$$
In particular, this means that $\tilde {\mathcal P}_{ D} 
\Pi_D {\mathcal P}_D w \in {\mathcal H} (\hat {\mathbb C}^n \setminus D)$ and  
$$
\overline \partial {\mathcal P}_D \tilde {\mathcal P}_{D} 
\Pi_D {\mathcal P}_D w =0 \mbox{ in } D \mbox{
for all } w\in {\mathcal H} (\hat {\mathbb C}^n \setminus D).
$$

Therefore, $v=\tilde {\mathcal P}_{ D} 
\Pi_D {\mathcal P}_D \hat v $ belongs to $ \Sigma (\hat {\mathbb C}^n \setminus D)$ 
and, according to \eqref{eq.limit}, \eqref{eq.f.limit}, 
\begin{equation*} 
f(u)  = -\lim_{\nu  \to \infty}  h_{D} (u_\nu, v)   =
\end{equation*}
\begin{equation*}
\int_{\partial G} \sum_{j=1}^n  (-1)^{j-1}
\overline{(\overline \partial_j v (\zeta))} \, u(\zeta) \,     d\overline \zeta 
 [j]\wedge d\zeta    = \langle u, v \rangle,
\end{equation*}
i.e. mapping \eqref{eq.mapping.compl} is surjective. 

Finally, the Closed Graph Theorem for the $DF$-spaces, see \cite[Ch. 6]{Edw} or 
\cite[Corollary A.6.4]{Mori}, implies that the inverse mapping for \eqref{eq.mapping.compl} 
is continuous, too, i.e. mapping \eqref{eq.mapping.compl} is a topological (anti-linear) 
isomorphism. 

The proof of Theorem \ref{t.dual.compl.over.G} is complete.

\section{Miscellaneous } 
\label{s.misc}

As we have noted in the Introduction, for a bounded domain $D \Subset {\mathbb C}^n$ with 
real analytic boundary, such that the space ${\mathcal O}(\overline D)$  is dense  in 
${\mathcal O}(D)$, Theorem \ref{t.dual.compl.over.G} can be extracted from the results of 
\cite{NaciShlaTark1}, where the duality \eqref{eq.dual.self} was proved for a very special 
pairing related to \eqref{eq.pair.A.compl.NShT} (cf. also \cite{AizeGind1}, \cite{AizeMit}, 
\cite{Stou1}, and  \cite{ShTaDBRK} for \eqref{eq.dual.self} corresponding to some other types 
of pairings in the bounded domain with real analytic boundary). The essential difference between the assumptions on $D$ in this particular case appears because the proof in \cite{NaciShlaTark1} uses the Grothendieck duality 
$$
({\mathcal H} (D ))^* \cong {\mathcal H} (\hat {\mathbb C}^{n} \setminus  D )  
$$
following from \cite[Theorem 4]{Grot1} for any bounded domain and  the continuous mapping 
$$
\tilde {\mathcal P}_{D} : 
{\mathcal O} (\overline D)  \to {\mathcal H} (\hat {\mathbb C}^{n} \setminus  D ) 
$$
that is granted by \cite{MorrNire1} for domains with real analytic boundaries (of course, 
this can not be true for all bounded domains even with $C^\infty$-smooth boundaries).

It follows from \cite[Theorem 8.1]{NaciShlaTark1} that pairing defined in 
\cite{NaciShlaTark1} induces a topological isomorphism \eqref{eq.dual.self} if and only if 
he space ${\mathcal O} (\overline D)$ is dense in ${\mathcal O}(D)$. The last observation 
suggests us to look for a similar statement related to mapping  \eqref{eq.mapping.compl}. 

\begin{corollary} \label{c.dual.compl.over.G}
Let $D\subset {\mathbb C}^n$, $n>1$, be a bounded domain with a connected Lispschitz boundary. Then
pairing \eqref{eq.pair.A.compl.NShT} induces a topological isomorphism 
\eqref{eq.dual.A.compl.over.G} if and only if the space $H^1 (D) \cap {\mathcal O} (D)$ 
is dense in ${\mathcal O} (D)$. 
\end{corollary}

\begin{proof}
Taking into the account Theorem \ref{t.dual.compl.over.G}, we are to prove 
that $H^1 (D) \cap {\mathcal O}(D) $ is dense in ${\mathcal O}(D) $.
With this purpose, let $F$ be a continuous linear functional on ${\mathcal O}(D) $ vanishing on
$H^1 (D) \cap {\mathcal O}(D) $. By the Hahn-Banach theorem, our statement will be proved once
we show that $F\equiv 0$. 
By assumption, there is a $v \in \Sigma (\hat {\mathbb C}^n \setminus D)$
such that $\langle v, u\rangle =0$ for all $u \in H^1 (D) \cap {\mathcal O}(D) $.
 It follows from Corollary \ref{c.Sigma} that ${\mathcal P}_D v \in 
H^1 (D) \cap {\mathcal O}(D)$ and hence 
$\langle v, {\mathcal P}_D v\rangle =0$. 
Thus, an argument similar to that in the proof of the injectivity 
in Theorem \ref{t.dual.compl.over.G} shows that $v = 0$ 
in $D$. Hence $F= 0$, as desired. 
\end{proof}

\smallskip

\textit{Acknowledgments\,} The work was supported by Krasnoyarsk Mathematical Center
funded by the Ministry of Education and Science of the Russian Federation 
(Agreement 075-02-2022-876).


\begin{thebibliography}{XXXXXX}

\bibitem{Aize3}
{ Aizenberg L.~A.},
  {\it The general form of a continuous linear functional in spaces of
       functions holomorphic in convex domains in ${\mathbb C}^{n}$},
  Dokl. Akad. Nauk SSSR {\bf 166} (1966), no. 5, 1015--1018.

\bibitem{AizeGind1}
{Aizenberg  L.~A.}, {Gindikin S.~G.},
  {\it On the general form of a linear continuous functional in spaces of
       holomorphic functions},  Uchen. Zap. 
  Moskov. Oblast. Ped. Inst. {\bf 137} (1964), 7--15.

\bibitem{AizeMit}
Aizenberg L.~A.,  Mityagin B.~S., {\it 
Spaces of functions analytic in multi-circular domains}, 
Sib. Math. Journal, 1:2 (1960), 153--170.

\bibitem{Bla}	
Blanchet P.,  {\it A duality theorem for solutions of elliptic equations}, 
Internat. J. Math. and Math. Sci. V. 13, no. (1990), 73--86.

\bibitem{Cost88} Costabel, M., 
Boundary integral operators on Lipschitz domains:
elementary results, SIAM J. Math. Anal., Vol. 19, No. 3,  1988, 613--626. 

\bibitem{Edw} Edwards, R.E. Functional Analysis. Theory and Applications. London, 
Holt, Rinehart and Winstone Inc., 1965. 

\bibitem{GuRo} Gunning R.C., Rossi, H., Analytic functions of several complex variables, 
Englewood Cliffs, N.J., Prentice Hall, Inc., 1965.

\bibitem{GiTru83} 
Gilbarg, D., Trudinger, N., \textit{Elliptic Partial Differential 
Equations of second order}, Berlin, Springer-Verlag, 1983.

	\bibitem{Grot2}	
{Grothendieck A.},  {\it Sur certain espaces de fonctions holomorphes. I}, 
J. Reine Angew. Math., 	192 (1953) pp. 35--64. 
	
	\bibitem{Grot1}
{Grothendieck A.},
  {\it Sur les espaces de solutions d'une classe generale d'equations aux
       derivees partielles},   J. Anal. Math. {\bf 2} (1953), 243--280.
	
	\bibitem{KohnJJ}
	Kohn J.J., {\it Subellipticity of the $\overline \partial$-Neumann problem on pseudo-convex domains: sufficient conditions}, Acta Math. 142 1--2 (1979), 79--122.

	\bibitem{Koth2}
{K{\"o}the G.}, {\it 
Dualit\"at in der Funktionentheorie}, J. Reine Angew. Math. , 191 (1953), 30--39.


\bibitem{Ky} Kytmanov A.M., 
The Bochner-Martinelli Integral and Its Applications, Birkh\"auser, Basel, 
1995.

\bibitem{MantSpag1}
Mantovani F., Spagnolo S., 
{\it Funzionali analitici e funzioni armoniche}. 
\newblock {Ann. Scuola Normale Sup. Pisa}, 18 (1964), 475--512.

\bibitem{Mart1}
{Martineau A.},
{\it Sur la topologies des espaces de fonctions holomorphes}, Math. Ann., 163 (1966), 62--88.

\bibitem{Mart2}
{Martineau A.},
  {\it Sur les fonctionelles analytiques et le transformation de
       {F}ourier-{B}orel},
  J. Anal. Math. {\bf 9} (1963), 1--164.

\bibitem{MorrNire1}
{ Morrey C.~B.},  { Nirenberg L.},
  {\it On the analyticity of the solutions of linear elliptic systems of
       partial differential equations},
  Comm. Pure and Appl. Math. {\bf 10} (1957), 271--290.

\bibitem{Mori}
Morimoto, M. An Introduction to Sato's Hyperfunctions. AMS, Providence,
Rhode Island, 1993.
	
\bibitem{NaSh}	
Nacinovich M., Shlapunov A.A., 
{\it On iterations of Green
integrals and their applica\-ti\-ons to elliptic differential complexes}, 
Mathematische Nachrichten, 180 (1996),  243--284.
	
	\bibitem{NaciShlaTark1}
{Nacinovich M., Shlapunov A.}, { Tarkhanov N.},
 {\it Duality in the spaces of solutions of elliptic systems},
 Ann. Scuola Norm. Sup. Pisa, {XXVI} (1998), no. 4, 207--232.

	\bibitem{NakSa}
Nakai M., Sario L.  {\it  Harmonic Functionals on Open Riemann Surfaces}, Pac. J.
Math. 93 (1981), 147--161.


\bibitem{Nap}
Napalkov V. V. (Jr.), 
{\it Various representation of the space of analytic functions and the problem of 
the dual space description}, Doklady Mathematics, V.66, no. 3 (2002), 335--337.

\bibitem{Rom}
Romanov A.V., {\it Convergence of iterates of the Bochner--Martinelli operator, and the 
Cauchy--Riemann equation}, Soviet Math. Dokl., 19:5 (1978), 1211--1215.

\bibitem{ReSh} Schulze B.-W., Rempel S., 
Index theory  of elliptic boundary problems, Akademie-Verlag, Berlin, 1982.

\bibitem{Shaef} 
Schaefer H.H.,  Wolff M.P., 
Topological Vector Spaces,  2nd Edition,  Springer, Berlin, 1999.

\bibitem{Se} 
Serre J.P., {\it Une theor\'eme de dualit\'e}, Comment. Math. Helvetici, 29 (1955), 9--26

\bibitem{Silva}
Sebasti\~{a}o e Silva J., {\it Analytic functions in functional analysis}, 
Portug. Math., 9 (1950), 1--130.


	\bibitem{ShDualSMJ}
	Shlapunov A.A., 
{\it Duality in spaces of solutions to elliptic systems}, 
Siberian Math. Journal, V.~43, N.~4 (2002),  948--958.

\bibitem{ShTaSMJ}
Shlapunov A.A, Tarkhanov N.N., 
{\it On the Cauchy problem for holomorphic functions of Lebesgue class  $L^2$  in  domains.} 
Siberian Math. Journal, V.~33, N.~5 (1992),  914--922.

\bibitem{ShTaDBRK}
Shlapunov A.A., Tarkhanov N., 
{\it  Duality by reproducing kernels}. 
International Journal of Math. and Math. 
Sciences, {\bf 6} (2003), 78pp.

\bibitem{So} Sobolev S.L.
{ Introduction to 
the Theory of Cubature Formulas}, Nauka, Moscow, 1974.

\bibitem{Stou1}
{Stout E.~L.},
  {\it Harmonic duality, hyperfunctions and removable singularities},
  Izv. RAN, Ser. math. {\bf 59} (1995), no. 6, 133--170.

\bibitem{Tark35}
{Tarkhanov N.},
  {\it Complexes of Differential Operators},
  Kluwer Academic Publishers, Dordrecht, NL, 1995.

\bibitem{Tark45}
{Tarkhanov N.},
  {\it The {A}nalysis of {S}olutions of {E}lliptic {E}quations},
  Kluwer Academic Publishers, Dordrecht, NL, 1997.

\bibitem{Vla} Vladimirov V.S., Equations of the Mathematical Physics, 
Nauka, Moscow, 1988. 

\bibitem{Zorn1}
{Zorn P.},
  {\it Analytic Functionals and Bergman spaces},
  Ann. Scuola Norm. Sup. Pisa {\bf IX} (1982), 365--404.
	
\end{thebibliography}
\end{document}